\documentclass[12pt]{amsart}

\usepackage{amsmath}
\usepackage{amsfonts}
\usepackage{amssymb}
\usepackage{amsthm}
\usepackage{hyperref}
\usepackage{enumerate}
\usepackage{cleveref}
\usepackage{mathrsfs}
\usepackage[top=0.9in, bottom=1.15in, left=1.1in, right=1.1in]{geometry}
 \usepackage{hyperref}
\hypersetup{colorlinks=true,linkcolor=blue,citecolor=magenta}
\newtheorem{theorem}{Theorem}

\newtheorem*{remark}{Remark}

\Crefname{conjecture}{Conjecture}{Conjectures}

\theoremstyle{plain}

\theoremstyle{plain}

\author[Schneider]{Robert Schneider}
\address{Department of Mathematics\newline
University of Georgia\newline
Athens, Georgia 30602, U.S.A.}
\email{robert.schneider@uga.edu}

\title{Nuclear partitions and a formula for $p(n)$}
 
\begin{document}
 
\begin{abstract}
Define a ``nuclear partition'' to be an integer partition with no part equal to one. In this study we prove a simple formula to compute the partition function $p(n)$ by counting only the nuclear partitions of $n$, a vanishingly small subset by comparison with all partitions of $n$ as $n\to \infty$. Variations on the proof yield other formulas for $p(n)$, as well as Ramanujan-like congruences and an application to parity of the partition function.  \end{abstract}

\maketitle

\section{Nuclear partitions}
Let $\mathcal P$ denote the set of integer partitions.\footnote{See Andrews \cite{Andrews_theory} for further reading about partitions.} Here we count partitions of $n$ 
via a natural subclass of partitions we refer to as {\it nuclear partitions}, which are partitions having no part equal to one. 
Let us denote the nuclear partitions by $\mathcal N\subset \mathcal P$, 
and let $\mathcal N_n$ denote nuclear partitions of $n\geq 0$. 
We use atomic terminology because partitions in $\mathcal N$ generate the rest of the set $\mathcal P$ through an algorithm resembling nuclear decay, which we detail below. 

Let $p(n)$ denote the {\it partition function}, i.e., the number of partitions of $n\geq 0$, where $p(0):=1$. This important arithmetic function grows rapidly but somewhat irregularly; due to  deep connections in algebra and physics, as well as number theory and combinatorics, intense interest surrounds its behavior. While Euler's product-sum generating function for the partition function and his recursive pentagonal number theorem both give ways to produce the sequence of values $p(n)$ computationally (see \cite{Andrews_theory}), an explicit formula quantifying the complicated behavior of the partition function was out of reach until 1938 when Rademacher, extending an asymptotic formula of Hardy and Ramanujan, proved an infinite series exactly equal to $p(n)$ \cite{Rademacher}. In 2007, Bringmann-Ono proved a beautiful {\it finite} formula for $p(n)$ as a sum of algebraic integers \cite{BO}, and Choliy-Sills in 2016 gave a purely combinatorial finite formula for the partition function \cite{CS}.  While we do not present closed formulas of this type, here we extract information to compute $p(n)$ from the nuclear partitions of $n$ aside from the partition $(n)$ into just one part, i.e., the set $\mathcal N_n\backslash (n)$, which has significantly fewer elements than the set $\mathcal P_n$ of partitions of $n$. 

For $n\geq 1$, let $\nu(n)$ count the number of nuclear partitions of $n$, with $\nu(0):=1$ (noting $\nu(1)=0$). Clearly we have the recursive relation $p(n)=\nu(n)+p(n-1)$.\footnote{$\nu(n)$ represents the first difference of $p(n)$. Thus $\nu(n)$ has the generating function ${(q^2;q)_{\infty}^{-1}}$, where $(z;q)_{\infty}:=\prod_{n=0}^{\infty}(1-z q^{n}),\  z,q\in \mathbb C, |q|<1.$} Then 
\begin{equation}\label{Achain}
p(n)=\nu(0)+\nu(1)+\nu(2)+\nu(3)+\cdots+\nu(n).
\end{equation}
So to count partitions of $n$, we need only enumerate nuclear partitions with sizes up to $n$. 

Let $\mu=(\mu_1,\mu_2,\dots,\mu_r),\  \mu_1\geq\mu_2\geq \dots \geq  \mu_r \geq 2$, denote a nuclear partition. 

\begin{theorem}\label{Acountingthm}
We have that
$$p(n)=n+\nu(n)-1+\sum_{\mu\in \mathcal N_n\backslash(n)}(\mu_1-\mu_2),$$
with the right-hand sum taken over nuclear partitions of $n$ apart from the partition $(n)$.
\end{theorem}

To compute $p(n)$ from Theorem \ref{Acountingthm}, one can follow these steps: 

\begin{enumerate}[1.]
\item Write down the partitions of $n$ containing no 1's aside from $(n)$ itself, that is, the subset $\mathcal N_n \backslash (n)$. This is a relatively small subset of the partitions of $n$, as we prove below. For example, to find $p(6)$ we use $\mathcal N_6\backslash(6)=\{(4,2), (3,3), (2,2,2)\}$.
\item Write down the difference $\mu_1-\mu_2\geq 0$ between the first part and the second part of each partition from the preceding step. In the present example, we write down $$4-2=2,\  \  \  \  3-3=0,\  \  \  \  2-2=0.$$
\item Add together the differences obtained in the previous step, then add the result to $n+\nu(n)-1$ to arrive at $p(n)$. In this example, we add $2+0+0=2$ from the previous step to $6+\nu(6)-1=6+4-1$, arriving at $p(6)=6+4-1+2=11$, which of course is correct.  
\end{enumerate}

\begin{proof}[Proof of Theorem \ref{Acountingthm}]
Let $m_k(\lambda)$ denote the {\it multiplicity} of $k\geq 1$ as a part of partition $\lambda=(\lambda_1, \lambda_2,\dots, \lambda_r),\  \lambda_1\geq \lambda_2\geq \dots \geq \lambda_r\geq 1$. Observe that every nuclear partition of $n$ can be formed by adding $m_1(\lambda)$ to the largest part $\lambda_1$ of a ``non-nuclear'' partition $\lambda\vdash n$, and deleting all the 1's from $\lambda$, e.g., $(3,2,1,1)\to (5,2)$. Conversely, every nuclear partition $\mu \vdash n$ can be turned into a non-nuclear partition of $n$ by decreasing the largest part $\mu_1$ by some positive integer $j\leq \mu_1-\mu_2$, and adjoining $j$ 1's to form the non-nuclear partition, or else $ \mu_1-\mu_2=0$ which we think of as the ``ground state''. 

So (non-ground state) nuclear partitions of $n$ ``decay'', by giving up 1's from the largest part, into non-nuclear partitions of $n$, e.g., $(5,2)\to(4,2,1)\to (3,2,1,1)\to(2,2,1,1,1)$, of which the total number is $p(n)-\nu(n)$. Each nuclear partition $\mu$ decays into $\mu_1-\mu_2$ different non-nuclear partitions except the partition $(n)$, which decays into $n-1$ non-nuclear partitions, viz. $(n)\to(n-1,1)\to (n-2,1,1)\to \dots \to (1,1,1,\dots,1)$. Therefore, the number of non-nuclear partitions of $n$ is $p(n)-\nu(n)=(n-1)+\sum_{\mu\in \mathcal N_n\backslash(n)}(\mu_1-\mu_2).$
\end{proof}

It is interesting to observe how the subset $\mathcal N\subset \mathcal P$ produces the entire set $\mathcal P$ by the ``decay'' process described in the proof above. To prove this, assume otherwise, that for some $n\geq 0$ there is a non-nuclear partition $\phi=(\phi_1, \phi_2, \dots, \phi_r)$ of $n$ that is {\it not} produced by the decay of some partition in $\mathcal N_n$. Then deleting all the $1$'s from $\phi$ and adding them to the largest part $\phi_1$ produces a nuclear partition of $n$ that decays into $\phi$, a contradiction. 

One can see that the subset $\mathcal N_n$ indeed grows much more slowly than $\mathcal P_n$ from the Hardy-Ramanujan asymptotic for the partition function (see \cite{Andrews_theory}),
\begin{equation}\label{H-R}
p(n) \sim \frac{e^{A\sqrt{n}}}{Bn}\  \  \text{with} \  \  A=\pi\sqrt{2/3},\  \  B=4\sqrt{3}.
\end{equation}
Applying \eqref{H-R} in the relation $\nu(n)=p(n)-p(n-1)$ gives, after a little algebra,
\begin{equation}\label{H-R2}
\nu(n)\  \sim\  \frac{e^{A\sqrt{n}}\left( 1- e^{-A\left(\sqrt{n}-\sqrt{n-1} \right)} \right)}{Bn}\  \sim\  \frac{A \cdot e^{A\sqrt{n}}\left(\sqrt{n}-\sqrt{n-1}\right)}{Bn} ,
\end{equation}
where on the right we use that $1-e^{-x} \sim x$ as $x\to0^+$. Then using \eqref{H-R} and \eqref{H-R2} to compute $\lim_{n\to \infty} \nu(n)/p(n)=0$, we see that $\nu(n) = o\left(p(n)\right)$.\footnote{More refined estimates for $\nu(n)$ can be found, e.g. see Eq. (1.2) of \cite{S}.} 

Please refer to Table 1 for an explicit comparison of the first few values of each function.

\section{Ground state nuclear partitions}
Let $\gamma(n)$ denote the number of nuclear partitions $\mu=(\mu_1, \mu_2, \dots, \mu_r)$ of $n$ such that 
$\mu_1=\mu_2$ (the first two parts are equal), 
setting $\gamma(0):=0$ and noting $\gamma(1)=\gamma(2)=\gamma(3)=0$, as well. As in the previous section, we refer to this even sparser subset $\mathcal G \subset \mathcal N$ as {\it ground state nuclear partitions}, with $\mathcal G_n$ denoting ground state nuclear partitions of $n$. In fact, the set $\mathcal N$ itself can be recovered from information about the subset $\mathcal G$.

\begin{table}
\vskip.1in  \begin{center}
\begin{tabular}{ | c | c | c | c | }
\hline
{\bf $n$} &{\bf $\gamma(n)$} & {\bf $\nu(n)$}&  {\bf $p(n)$}\\ \hline
1 & 0 & 0 & 1   \\ \hline
2 & 0 & 1 &  2   \\ \hline
3 & 0 & 1 & 3   \\ \hline
4 & 1 & 2 & 5  \\ \hline
5 & 0 & 2 & 7  \\ \hline
6 & 2 & 4 & 11   \\ \hline
7 & 0 & 4 & 15   \\ \hline
8 & 3 & 7 & 22 \\  \hline
9 & 1 & 8 & 30   \\  \hline
10 & 4 & 12 & 42  \\  \hline
11 & 2 & 14 & 56  \\  \hline
12 & 7 & 21 & 77  \\  \hline
13 & 3 & 24 & 101  \\  \hline
14 & 10 & 34 & 135  \\  \hline
15 & 7 & 41 & 176  \\  \hline
16 & 14 & 55 & 231  \\  \hline
17 & 11 & 66 & 297  \\  \hline
18 & 22 & 88 & 385  \\  \hline
19 & 17 & 105 & 490  \\  \hline
20 & 32 & 137 & 627  \\  \hline
$\cdots$ & $\cdots$ & $\cdots$ & $\cdots$ \\  \hline
$100$ & \  $2{,}307{,}678$\   & \  $21{,}339{,}417$\   & \  $190{,}569{,}292$\  \\
\hline
\end{tabular}
\label{Table}
\smallskip
\smallskip
\smallskip
\caption{Comparing the growth of $\mathcal G_n, \mathcal N_n$ and $\mathcal P_n$}\end{center}
\end{table}

Clearly we have for $n\geq 3$ the recursion $\nu(n)=\gamma(n)+\nu(n-1)$, viz. adding 1 to the largest part of every nuclear partition of $n-1$ gives the nuclear partitions $\mu$ of $n$ with $\mu_1>\mu_2$.\footnote{$\gamma(n)$ represents the second difference of $p(n)$, thus its generating function is $\frac{1}{(1+q)(q^3;q)_{\infty}}-1+q-q^2$.} Moreover, noting $\nu(2)=1$, we have for $n\geq 2$ that  
\begin{equation}\label{nu_recurrence} \nu(n)=1+\gamma(3)+\gamma(4)+\cdots +\gamma(n).
\end{equation} 
Much as nuclear partitions ``control'' the growth of $p(n)$, these ground state nuclear partitions control the growth of $\nu(n)$, and thus appear to fundamentally control $p(n)$:
\begin{flalign}\label{nu_rec2}
p(n)&=n+(n-2)\gamma(3)+(n-3)\gamma(4)+(n-4)\gamma(5)+\cdots + 2\gamma(n-1)+\gamma(n),
\end{flalign}
which comes from combining \eqref{Achain} and \eqref{nu_recurrence}. This computation involves generating substantially fewer than even the $\#\mathcal N_n$ partitions enumerated by $\nu(n)$, to arrive at $p(n)$ (although as $n$ increases this is still a nontrivial task, as the $n=100$ row of Table 1 reveals). One can see from \eqref{H-R2} and the recursion $\gamma(n)=\nu(n)-\nu(n-1)$ that in fact
\begin{flalign}\label{H-R3}
\gamma(n) &\sim \frac{e^{A\sqrt{n}}\left( e^{-A\left(\sqrt{n-1}-\sqrt{n-2} \right)} - e^{-A\left(\sqrt{n}-\sqrt{n-1}\right)} \right)}{Bn}\\ \nonumber \\
\nonumber 
&\sim \frac{A \cdot e^{A\sqrt{n}}\left(\sqrt{n}-2\sqrt{n-1}+\sqrt{n-2}\right)}{Bn}.
\end{flalign}
Using this asymptotic along with \eqref{H-R2} to compute  $\lim_{n\to \infty} \gamma(n)/\nu(n)$ (applying L'Hospital's rule in the variable $n$) gives that $\gamma(n) = o\left( \nu(n) \right)$; moreover, $\gamma(n)/p(n)\to 0$ quite rapidly.

For a more concrete look at the relative growths of $\gamma(n), \nu(n)$ and $p(n)$, the reader is referred to Table 1.\footnote{The author computed rows $n=1$ to $20$ of Table 1 by hand: first compute the $\gamma(n)$ values by generating ground state nuclear partitions, then use those values and the recursion $\nu(k+1)=\nu(k)+\gamma(k+1)$ to fill in the $\nu(n)$ column, and finally use those values and $p(k+1)=p(k)+\nu(k+1)$ to fill in the $p(n)$ column.}
Inspecting the columns of the table, an immediate difference between $\gamma, \nu$ and $p$ is that, while the latter two functions are always increasing (at least weakly) for $n\geq 1$, the function $\gamma(n)$ oscillates somewhat irregularly in magnitude 
 --- this contributes to the locally irregular fluctuations of the partition function.
 
For instance, one sees in Table 1 that certain integers $k\leq n$ with abnormally large values of $\gamma(k)$ by comparison to other nearby integers, make a larger impact on the growth of $p(n)$ in the table, causing locally bigger jumps in the value of the function. One such class appears to be the integers with abnormally many divisors, since each divisor $d$ of $k$, $d\neq 1$,  produces the new ground state nuclear partition $(d,d,\dots,d)$ with $k/d$ repetitions.\footnote{However, it is not obvious that this makes a significant impact on the behavior of $p(n)$ as $n\to \infty$, as Dennis Eichhorn noted to the author (personal communication, June 8, 2019).} Identifying integers having abnormally many (or few) ground state nuclear partitions could provide information about the behavior of $p(n)$.
 

\section{Further observations and directions for study}

For $m\geq 1$, let $\nu(n,m)$ denote the number of nuclear partitions of $n$ whose parts are all $\leq m$. Then it is easily verified that we can write $\nu(n)$ for $n\geq4$ as follows: 
\begin{equation}
\nu(n)=\sum_{k=2}^{n-2}\nu(k,n-k).
\end{equation}
Combining this identity with Theorem \ref{Acountingthm} and \eqref{nu_recurrence} above, and making further simplifications, the task of computing $p(n)$ can be reduced to counting much smaller subsets of ground state nuclear partitions of integers up to $n-2$. These small subsets $\mathcal G_k$ of partitions of integers $k \leq n-2$ completely encode the value of $p(n)$.

More generally, we might let $\nu_k(n)$ denote the number of partitions of $n$ having no part equal to $k$ --- let us refer to these as {\it $k$-nuclear partitions} --- 
setting $\nu_k(0):=1$ for all $k\geq 1$; thus $\nu(n)=\nu_1(n)$. Let $\mathcal N^k$ denote the set of all $k$-nuclear partitions, and let $ \mathcal N_n^k$ be $k$-nuclear partitions of $n$; thus $\mathcal N=\mathcal N^1, \  \mathcal N_n=\mathcal N_n^1$. Clearly we have $p(n)=\nu_k(n)+p(n-k)$, so $\nu_k(n)$ is subject to essentially the same treatment as $\nu(n)$.\footnote{Thus $\nu_k(n)$ has generating function $\frac{1-q^k}{(q;q)_{\infty}}$.} 

Then by recursion, as previously, we have 
\begin{equation} p(n)=p\left(n-\left\lfloor{\frac{n}{k}}\right\rfloor k\right)+\sum_{j=1}^{\left\lfloor{n/k}\right\rfloor} \nu_k(n-jk),\end{equation}
where $\left\lfloor{x}\right\rfloor$ is the floor function. One could likely then generalize Theorem \ref{Acountingthm} (the $k=1$ case) by carefully carrying out similar steps to the theorem's proof, but using decay into $k$'s instead of 1's, e.g. $(n) \to (n-k, k)\to (n-2k, k, k) \to (n-3k, k, k, k)  \to \cdots$.
%
%
%
%
 
Considering Table 1 again, certain numerical patterns stand out. For instance, based on even rows $n=2m$ in the table, it looks like $p(2m) \approx \sqrt{2m} \cdot \nu(2m)\approx 2m \cdot \gamma(2m)$, but this could be a coincidence due to the small number of entries available.\footnote{Indeed, it would be surprising, e.g. one expects $\nu(2m)/p(2m) \sim A\cdot\left(\sqrt{2m}-\sqrt{2m-1}\right)$ by \eqref{H-R}, \eqref{H-R2}.}  
We also see all three entries of row $n=12$ are divisible by 7, and the three entries of row $17$ are divisible by $11$; these numerics are reminiscent of the Ramanujan congruences \cite{Ramanujan}
\begin{equation}
p(5n+4)\equiv 0\  (\text{mod}\  5),\  \  \  \  \    p(7n+5)\equiv 0\  (\text{mod}\  7),\  \   \  \  \  p(11n+6)\equiv 0\  (\text{mod}\  11),
\end{equation} 
for nonnegative $n$. Since $p(5n+4)-p\left(5n-1\right)\equiv 0\  (\text{mod}\  5)$ for $n\geq 1$, then \eqref{Achain} gives
\begin{equation}\sum_{k=5n}^{5n+4}\nu(k)\equiv 0\  (\text{mod}\  5).\end{equation}
The same argument yields
\begin{equation}\nu_5(5n+4)\equiv 0\  (\text{mod}\  5),\end{equation}
and, likewise, the $p(7n+5)$ and $p(11n+6)$ congruences give  
 \begin{flalign}\sum_{k=7n}^{7n+5}\nu(k)\equiv 0\  (\text{mod}\  7),\  \  \  \  \  \    &\nu_7(7n+5)\equiv 0\  (\text{mod}\  7),\\ \sum_{k=11n}^{11n+6}\nu(k)\equiv 0\  (\text{mod}\  11),\  \  \  \  \  \    &\nu_{11}(11n+6)\equiv 0\  (\text{mod}\  11).\end{flalign}
Similar generalizations and congruences presumably extend to ground state nuclear partitions. It seems that proceeding in this direction could address the observations about all three congruences across rows $n=12,17$, along with more general cases.

Other congruences can be found readily; further combining the above formulas gives
\begin{equation}
\gamma(5n+1)+2\gamma(5n+2)+3\gamma(5n+3)+4\gamma(5n+4)\equiv 0\  (\text{mod}\  5),
\end{equation}
and likewise for the modulo 7 and 11 cases. If these kinds of congruences 
could be established directly, then working in the opposite direction, one might 
prove the Ramanujan 
congruences by induction.

We can also find the parity of $p(n)$ (and more generally, compute $p(n)$ modulo $m$, $m\geq 2$) from values of $\gamma(k),k\leq n$. For instance, one can rewrite \eqref{nu_rec2} in the form
\begin{equation}
p(n)=n\cdot \nu(n)-\sum_{k=3}^{n} (k-1)\gamma(k),
\end{equation}
which, for $n\geq 4$ an even number, reduces to
\begin{equation}\label{mod_2}
p(n)\equiv \gamma(4)+ \gamma(6) + \gamma(8) + \cdots  + \gamma(n-2) + \gamma(n)\  \   (\operatorname{mod} 2).
\end{equation} 
To compute an example using \eqref{mod_2}, from the $\gamma(n)$ column of Table 1 we find $p(20)$ is congruent to $1+2+3+4+7+10+14+22+32=95$ modulo 2, confirming $p(20)$ is odd. 

\begin{remark}
For further resources on first, second and higher-order differences of the partition function, see \cite{M,O}.
\end{remark}

\section*{Acknowledgments}
The author is grateful to Matthew Just for noting corrections in this paper, to Dennis Eichhorn for a helpful discussion about ground state nuclear partitions, and to Brian Hopkins and the anonymous referee for suggestions that greatly strengthened this work.

\end{document}